\documentclass[reqno]{amsart}
\usepackage[english]{babel}
\usepackage{amssymb}
\usepackage[T1,T5]{fontenc}
\newtheorem{theorem}{Theorem}[section]

\newtheorem{lemma}[theorem]{Lemma}

\theoremstyle{definition}
\newtheorem{definition}[theorem]{Definition}
\newtheorem{example}[theorem]{Example}
\theoremstyle{remark}
\newtheorem*{remark}{Remark}
\numberwithin{equation}{section}

\newcommand{\UU}{\operatorname{U}_{f}}
\newcommand{\HH}{\operatorname{H}_{m}}

\newcommand{\C}{\mbox{$\mathbb{C}$}}

\newcommand{\E}{\mbox{$\mathcal{E}$}}

\newcommand{\Jp}{\operatorname{J}_p}

\begin{document}

\title{On the regularity of the complex Hessian equation}

\author{Per \AA hag}
\address{Department of Mathematics and Mathematical Statistics\\ Ume\aa \ University\\ SE-901 87 Ume\aa \\ Sweden}
\email{per.ahag@umu.se}
\author{Rafa\l\ Czy{\.z}}
\address{Institute of Mathematics \\ Faculty of Mathematics and Computer Science \\ Jagiellonian University\\ \L ojasiewicza 6\\ 30-348 Krak\'ow\\ Poland}
\email{Rafal.Czyz@im.uj.edu.pl}\thanks{The second-named author was supported by the Priority Research Area SciMat under the program Excellence Initiative - Research University at the Jagiellonian University in Krak\'ow.}

\dedicatory{We raise our cups to Urban Cegrell, gone but not forgotten, gone but ever here. \\ Until we meet again in Valhalla!}


\begin{abstract}  This note aims to investigate the regularity of a solution to the Dirichlet problem for the complex Hessian equation, which has a density of the $m$-Hessian measure that belongs to $L^q$, for $q\leq\frac nm$.
\end{abstract}

\maketitle

\begin{center}\bf
\today
\end{center}

\section{Introduction}

Let $n\geq 2$, $1\leq m\leq n$. Consider the following Dirichlet problem for the complex Hessian equation in a bounded $m$-hyperconvex domain:
\begin{equation}\label{Intr1}
\HH(\UU)=fdV_{2n},
\end{equation}
or on a compact K\"ahler manifold:
\begin{equation}\label{Intr2}
\HH(\UU)=f\omega^n.
\end{equation}
Here $dV_{2n}$ is the Lebesgue measure in $\mathbb{R}^{2n}$, $\omega$ is a K\"ahler form on $X$ such that $\int_{X}\omega^n=1$, and  $f\in L^{\beta}$ is a
density function. When $\beta>\frac{n}{m}$, Dinew and Ko\l odziej~\cite{DK} proved that the solution $\UU$ is a continuous $m$-subharmonic functions. Later
on it was proved that the solution is in fact  H\"older continuous (see e.g.~\cite{B,Ch,KN,N}). Furthermore, if $\beta<\frac{n}{m}$, then the solution to~(\ref{Intr1}) need not to be bounded (see \cite{DK}), which is a significant contrast to the case when $m=n$ (see e.g.~\cite{CP,K1}).

Our aim is to consider the remaining singular cases, when $\beta \leq \frac mn $, to complete the overall picture. Under this assumption, we shall prove the regularity of a solution $u$ in the sense of being a member of Cegrell's energy class with bounded $(p,m)$-energy, $\mathcal E_{p, m}(\Omega)$. This is possible since both~(\ref{Intr1}), and~(\ref{Intr2}), are solvable within $\mathcal E_{p, m}(\Omega)$ (\cite{cegrell_pc,L3}). We say that a function is more regular if it belongs to $\mathcal E_ {p,m}(\Omega)$ for a larger $p>0$, and this is motivated by the fact that $\mathcal E_{p,m}(\Omega)\subset\mathcal E_{q,m}(\Omega)$, for $q<p$. In this note, we shall use the following auxiliary notation: for $\beta>1$, we say that a $m$-subharmonic function $u$ belongs to the class $\mathcal M_{\beta}$, if $\HH(u)=fdV_{2n}$, and $f\in L^{\beta}(\Omega)$. Our regularity result for~(\ref{Intr1}) is:

\bigskip

\noindent\textbf{Theorem~\ref{mt}.} \emph{Let $n\geq 2$, $1\leq m<n$, and assume that $\Omega\subset \mathbb C^n$ is a $m$-hyperconvex domain. For $1<\beta\leq\frac{n}{m}$, let $u\in \mathcal M_{\beta}$. Then the following holds:}
\begin{enumerate}\itemsep2mm
\item \emph{if $\beta=\frac{n}{m}$, then $u\in \mathcal E_{p,m}(\Omega)$ for all $p>0$. Furthermore, $u\in L^q(\Omega)$, for all $q>0$;}

\item \emph{if $\beta<\frac{n}{m}$, then $u\in \mathcal E_{p,m}(\Omega)$ for $p<\frac{(\beta-1)nm}{n-\beta m}$. Furthermore, $u\in L^q(\Omega)$, for all}
\[
0<q<\frac{\beta nm}{n-\beta m}\, .
\]
\end{enumerate}

\bigskip

In Example~\ref{examp1}, we provide an example that shows that the results in Theorem~\ref{mt} are sharp. Furthermore, in Example~\ref{examp2} we show that in the case $\beta=\frac nm$, the solution can be unbounded. The corresponding result of Theorem~\ref{mt} for the compact K\"ahler manifold case~(\ref{Intr2}), is stated and proved in Theorem~\ref{mt_Kahler}.

Commonly, stability results of equations of the type~(\ref{Intr1}) is done by estimating the supremum norm of continuous, or bounded solutions, $|\UU-\operatorname U_{f_j}|$, in terms of the $L^q$-norm of the difference of the densities, $\|f_j-f\|_q$. When it comes to the possibility of unbounded solutions, this type of estimation is not possible. Instead one can try to get weaker results, for examples $\operatorname U_{f_j}$ tends to $\UU$ in capacity, or in $L^q$-norm, when $f_j$ converges to $f$. We shall approach it differently, and consider the convergence in the quasimetric space $(\mathcal E_{p,m}(\Omega),\Jp)$. Here the quasimetric $\Jp$ is defined by:
\[
\Jp(u,v)=\left(\int_{\Omega}|u-v|^p(\HH(u)+\HH(v))\right)^\frac {1}{p+m},
\]
where $u,v\in \mathcal E_{p,m}(\Omega)$, $p>0$. In~\cite{ACquasi}, it was proved that $(\mathcal E_{p,m}(\Omega),\Jp)$ is a complete quasimetric space. Furthermore,  convergence in $(X,\Jp)$ implies convergence in capacity, and in $L^q$-norm, but the converse is not true. In this way, our approach is preferable. Our stability result states as follows.

\medskip

\noindent\textbf{Theorem~\ref{stab}.} \emph{Let $n\geq 2$, $1\leq m<n$, $1<\beta\leq\frac{n}{m}$, and assume that $\Omega\subset \mathbb C^n$ is a $m$-hyperconvex domain. Furthermore, let $f_j,f\in L^{\beta}(\Omega)$ be such that $\|f_j-f\|_{\beta}\to 0$, as $j\to \infty$. Then for $p<p_{\infty}$ it holds that $\Jp(\operatorname{U}_{f_j},\UU)\to 0$. Here,}
\[
p_{\infty}=\left\{
             \begin{array}{ll}
               \frac{(\beta-1)nm}{n-\beta m}, & \text { if } \ \beta<\frac{n}{m}\\[2mm]
               \infty, & \text { if } \ \beta=\frac{n}{m}.
             \end{array}
           \right.
\]
\emph{Moreover, there exist a constant $C$ depending only on $\|f\|_{\beta}$, $\beta$, $m$, $p$, and a constant $\gamma$ depending only on $\beta$, $m$, $p$ such that}
\[
\operatorname J_{p}(\operatorname{U}_{f_j},\UU)\leq C\|f_j-f\|_{\beta}^{\gamma}.
\]
\emph{The constant $\gamma$ can be taken as}
\[
\begin{cases}
\gamma=\frac {p}{m(p+m)} & \text{, when } p\leq\frac {m(\beta-1)}{\beta};\\[2mm]
\gamma\in \left(0,\frac {\beta-1}{(p+m)\beta}\right)=\left(0,\frac {n-m}{n(m+p)}\right) & \text{, when } p>\frac {m(\beta-1)}{\beta}=\frac {m(n-m)}{n} \text{ and }\beta=\frac nm;\\[2mm]
\gamma\in \left(0,\frac {\beta m(n+p)-n(m+p)}{m(p+m)(\beta(n+m)-n)}\right)  & \text{, when } p>\frac {m(\beta-1)}{\beta} \text{ and }\beta<\frac nm.
\end{cases}
\]

\section{Preliminaries}

  We start with the definition of $m$-subharmonic functions and the  complex Hessian operator.  Let $\Omega \subset \C^n$, $n\geq 2$, be a bounded domain, $1\leq m\leq n$, and define $\mathbb C_{(1,1)}$ to be the set of $(1,1)$-forms with constant coefficients. Then, set
\[
\Gamma_m=\left\{\alpha\in \mathbb C_{(1,1)}: \alpha\wedge (dd^c|z|^2)^{n-1}\geq 0, \dots , \alpha^m\wedge (dd^c|z|^2)^{n-m}\geq 0   \right\}\, .
\]

\begin{definition}\label{m-sh} Let $n\geq 2$, and $1\leq m\leq n$. Assume that $\Omega \subset \C^n$ is a bounded domain, and let $u$ be a subharmonic function defined on $\Omega$. Then we say that $u$ is \emph{$m$-subharmonic} if the following inequality holds
\[
dd^cu\wedge\alpha_1\wedge\dots\wedge\alpha_{m-1}\wedge(dd^c|z|^2)^{n-m}\geq 0\, ,
\]
in the sense of currents for all $\alpha_1,\ldots,\alpha_{m-1}\in \Gamma_m$. With $\mathcal{SH}_m(\Omega)$ we denote the set of all $m$-subharmonic functions defined on $\Omega$.
\end{definition}

The real counterpart of $m$-subharmonic functions was first introduced by Caffarelli, Nirenberg, and Spruck~\cite{CNS85}. The origin of these functions in the complex setting we focus on here was defined by Vinacua~\cite{vinacua1,vinacua2}. Later, B\l ocki~\cite{Blocki} generalized this concept to the unbounded case we see in Definition~\ref{m-sh}, and pluripotential methods were introduced. For those who want more background on $m$-subharmonic functions, we refer to~\cite{SA,AS2,DK2,L}.

\begin{definition}\label{prel_hcx} Let $n\geq 2$, and $1\leq m\leq n$. A bounded domain in $\Omega\subset\C^n$ is said to be  \emph{$m$-hyperconvex} if it admits a non-negative and $m$-subharmonic exhaustion function, i.e.  there exits a $m$-subharmonic $\varphi:\Omega\to [0,\infty)$ such that the closure of the set $\{z\in\Omega : \varphi(z)<c\}$ is compact in $\Omega$, for every $c\in (-\infty, 0)$.
\end{definition}

For further information on  $m$-hyperconvex domains see~\cite{ACH}. Next, we shall recall the function classes that are of our interest. The following notations shall be used
\[
e_{0,m}(u)=\int_{\Omega} \operatorname H_m(u) \quad \text{ and } \quad e_{p,m}(u)=\int_{\Omega} (-u)^p \operatorname H_m(u).
\]
We say that a $m$-subharmonic function $\varphi$ defined on a $m$-hyperconvex domain  $\Omega$ belongs to $\mathcal E^0_{m}(\Omega)$ if $\varphi$ is bounded,
\[
\lim_{z\rightarrow\xi} \varphi (z)=0 \quad \text{ for every } \xi\in\partial\Omega\, ,
\]
and
\[
\int_{\Omega} \operatorname{H}_m(\varphi)<\infty\, .
\]
\begin{definition} Let $n\geq 2$, $1\leq m\leq n$, and $p\geq0$. Assume that $\Omega$ is a bounded $m$-hyperconvex domain in $\mathbb C^n$.
We say that $u\in \E_{p,m}(\Omega)$ for $p\geq 0$, if $u$ is a $m$-subharmonic function defined on $\Omega$ such that there exists a decreasing sequence, $\{\varphi_{j}\}$, $\varphi_{j}\in\mathcal E^0_{m}(\Omega)$, that converges pointwise to $u$ on $\Omega$,
as $j$ tends to $\infty$, and $\sup_{j} e_{p,m}(\varphi_j)< \infty$. Another common notation of the case $p=0$, $\E_{0,m}(\Omega)$, is $\mathcal{F}_{m}(\Omega)$.
\end{definition}

In~\cite{L,L3}, it was proved that for $u\in \E_{p,m}(\Omega)$ the complex Hessian operator, $\operatorname H_m(u)$, is well-defined, and
\[
\operatorname{H}_m(u)=(dd^cu)^m\wedge (dd^c|z|^2)^{n-m}\, ,
\]
where $d=\partial +\bar{\partial}$, and $d^c=\sqrt{-1}\big(\bar{\partial}-\partial\big)$. Theorem~\ref{thm_holder} is essential when working with $\E_{p,m}(\Omega)$, $p>0$ (see Lu~\cite{L,L3}, and Nguyen~\cite{thien}).

\begin{theorem}\label{thm_holder} Let $n\geq 2$, $1\leq m\leq n$, and $p>0$. Assume that $\Omega$ is a bounded $m$-hyperconvex domain in $\mathbb C^n$. For $u_0,u_1,\ldots , u_m\in\E_{p,m}(\Omega)$, we have
\begin{multline*}
\int_\Omega (-u_0)^p dd^c u_1\wedge\cdots\wedge dd^c u_m\wedge (dd^c|z|^2)^{n-m}\\ \leq
C\; e_p(u_0)^{p/(p+m)}e_p(u_1)^{1/(p+m)}\cdots
e_p(u_m)^{1/(p+m)}\, ,
\end{multline*}
where $C\geq 1$ depends only on $p,m,n$ and $\Omega$.
\end{theorem}

\section{The main result}\label{ms}

In this section we shall prove the regularity of the solution, $\UU$, to the following Dirichlet problem for the complex Hessian equation in a bounded $m$-hyperconvex domain $\Omega$,
\[
\HH(\UU)=fdV_{2n}.
\]
It was proved by Dinew and Ko\l odziej~\cite{DK}, that if $\beta>\frac nm$, then $\UU$ is continuous $m$-subharmonic functions, so $\UU\in \mathcal E_{q,m}(\Omega)$, for all $q>0$. Therefore, we shall focus on the case $\beta\leq\frac nm$. Note that since $\HH(\UU)$ is bounded and if $\UU\in \mathcal E_{p,m}(\Omega)$, then $\UU\in \mathcal E_{q,m}(\Omega)$, for $q\leq p$.

\begin{definition}
Let $\beta>1$. We say that $m$-subharmonic function $u$ belongs to the class $\mathcal M_{\beta}$ if $\HH(u)=fdV_{2n}$ and $f\in L^{\beta}(\Omega)$.
\end{definition}

\begin{lemma}\label{l1}
Let $\Omega$ be bounded $m$-hyperconvex domain. If $u\in \mathcal E_{p,m}(\Omega)\cap \mathcal M_{\beta}$, $p\geq 0$, then $u\in \mathcal E_{q,m}(\Omega)$, for
\[
\begin{cases}
0<q<\frac {(\beta-1)n(m+p)}{\beta(n-m)}, & \text{ when }\beta> \frac {n(m+p)}{m(n+p)};\\[2mm]
0<q\leq p,  & \text{ when } \beta\leq \frac {n(m+p)}{m(n+p)}.
\end{cases}
\]
Moreover, there exists a constant $D>0$, which does not depend on $u$, such that
\[
e_{q,m}(u)\leq D \|f\|_{\beta}^{\frac {m+q}{m}}.
\]
\end{lemma}
\begin{proof}
Let $u_j$ be an approximating sequence defined by  $\HH(u_j)=\min(j,f)dV_{2n}=f_jdV_{2n}$, $u_j\in \mathcal E^0_m(\Omega)\cap L^{\infty}(\Omega)$, and $u_j\searrow u$,
as $j\to\infty$. First assume that $\beta> \frac {n(m+p)}{m(n+p)}$, i.e. $\frac {(\beta-1)n(m+p)}{\beta(n-m)}> p$. For $\frac {1}{\alpha}+\frac {1}{\beta}=1$, we get
\begin{multline}\label{eq1}
\int_{\Omega}(-u_j)^q\HH(u_j)=\int_{\Omega}(-u_j)^qf_jdV_{2n}\\ \leq \left(\int_{\Omega}(-u_j)^{q\alpha}dV_{2n}\right)^{\frac {1}{\alpha}}\left(\int_{\Omega}f_j^{\beta}dV_{2n}\right)^{\frac {1}{\beta}},
\end{multline}
which is finite by the Sobolev type inequality (\cite[Theorem 5.4]{AC-P}) if $\alpha q<\frac {n(m+p)}{n-m}$. Hence,
\[
0<q<\frac {(\beta-1)n(m+p)}{\beta(n-m)}.
\]
By letting $j\to\infty$ we get desired result. Again, by the Sobolev type inequality (\cite[Theorem 5.4]{AC-P}) we have
\[
\|u\|_{\alpha q}\leq Ce_{q,m}(u)^{\frac {1}{m+q}},
\]
where the constant $C$ does not depend on $u$. Therefore, (\ref{eq1}) yields
\[
e_q(u)\leq C^{q}e_q(u)^{\frac {q}{m+q}}\|f\|_{\beta},
\]
and we arrive at
\[
e_{q,m}(u)\leq D \|f\|_{\beta}^{\frac {m+q}{q}}.
\]

 Now assume that $\beta\leq  \frac {n(m+p)}{m(n+p)}$, i.e. $\frac {(\beta-1)n(m+p)}{\beta(n-m)}\leq  p$. By repeating the above argument we get $u\in \mathcal E_{q,m}(\Omega)$ for $0<q\leq \frac {(\beta-1)n(m+p)}{\beta(n-m)}$, and for $p=q$. Now we can use a standard argument to show that actually $u\in \mathcal E_{q,m}(\Omega)$ when $0<q\leq p$.

\end{proof}

By Lemma~\ref{l1} we see that the only interested case for us is when $\beta\in \left(\frac {n(m+p)}{m(n+p)},\frac nm\right]$. Under this assumption we will in the next lemma improve Lemma~\ref{l1} further.

\begin{lemma}\label{l2}
Let $\Omega$ be bounded $m$-hyperconvex domain, and assume that $\beta\in \left(\frac {n(m+p)}{m(n+p)},\frac nm\right]$. If $u\in \mathcal E_{p,m}(\Omega)\cap \mathcal M_{\beta}$, then $u\in \mathcal E_{q,m}(\Omega)$, for
\[
\begin{cases}
0<q<\frac {(\beta-1)nm}{n-\beta m}, & \text{ when }\beta<\frac nm;\\[2mm]
q>0,  & \text{ when } \beta=\frac nm.
\end{cases}
\]
\end{lemma}
\begin{proof} We base this proof on the following iterating procedure: When $\beta\in \left(\frac {n(m+p)}{m(n+p)},\frac nm\right]$, then if $u\in \mathcal E_{p,m}(\Omega)$, then $u\in \mathcal E_{q,m}(\Omega)$, for $q<\frac {(\beta-1)n(m+p)}{\beta (n-m)}=p'(>p)$. Moreover, if  $\beta>\frac {n(m+p)}{m(n+p)}$,  then $\beta>\frac {n(m+p')}{m(n+p')}$, and we can continue repeating the  procedure. Again, if $u\in \mathcal E_{p',m}(\Omega)$, then $u\in \mathcal E_{q',m}(\Omega)$, for $q'<\frac {(\beta-1)n(m+p')}{\beta (n-m)}$. In this way, in each step, the new $q'$ is slightly bigger than $p'$. Thus, we obtain the following sequence
\[
p_N=\alpha(m+p_{N-1}), \ \ \text {where} \ \ \alpha=\frac {(\beta-1)n}{\beta(n-m)}.
\]
If $\alpha=1$, then $\beta=\frac nm$ and $p_N=m+p_{N-1}$. Hence, $p_N\to \infty$, so $u\in \mathcal E_{p,m}(\Omega)$ for all $p>0$. Now assume that $\alpha<1$. We have that the above sequence is convergent, since it is increasing and bounded from above by $p_{\infty}=\frac {(\beta-1)nm}{n-\beta m}$. Therefore, we get $p_N\to p_{\infty}=\frac {(\beta-1)nm}{n-\beta m}$, as $N\to \infty$. Hence, $u\in \mathcal E_{q,m}(\Omega)$ for $q<p_{\infty}$.
\end{proof}
\begin{remark}
Note that the exponent $q$ obtained in Lemma~\ref{l2} is better than the exponent obtained in Lemma~\ref{l1}, since
\[
\text{if }  \beta>\frac {n(m+p)}{m(n+p)},  \text { then } \  \frac {(\beta-1)n(m+p)}{\beta(n-m)}<\frac {(\beta-1)nm}{n-\beta m}.
\]
\end{remark}

\begin{theorem}\label{mt}  Let $n\geq 2$, $1\leq m<n$, and assume that $\Omega\subset \mathbb C^n$ is a $m$-hyperconvex domain. For $1<\beta\leq\frac{n}{m}$, let $u\in \mathcal M_{\beta}$. Then the following holds:
\begin{enumerate}\itemsep2mm
\item if $\beta=\frac{n}{m}$, then $u\in \mathcal E_{p,m}(\Omega)$ for all $p>0$. Furthermore, $u\in L^q(\Omega)$, for all $q>0$;

\item if $\beta<\frac{n}{m}$, then $u\in \mathcal E_{p,m}(\Omega)$ for $p<\frac{(\beta-1)nm}{n-\beta m}$. Furthermore, $u\in L^q(\Omega)$, for all
\[
0<q<\frac{\beta nm}{n-\beta m}\, .
\]
\end{enumerate}
\end{theorem}
\begin{proof}
Let $u\in \mathcal M_{\beta}$, then $u\in \mathcal E_{0,m}$, since $\HH(u)$ is bounded. Now we can use Lemma~\ref{l1} to get $u\in \mathcal E_{p,m}(\Omega)$, for $p<\frac {(\beta-1)nm}{\beta(n-m)}=p_0$. Observe that $\beta>\frac {n(m+p_0)}{m(n+p_0)}$. Now by Lemma~\ref{l2} we get the thesis. The integrability condition for $u$ follows from~\cite{AC-P}.
\end{proof}

In Example~\ref{examp1} we provide an example that shows that our result is sharp, and in Example~\ref{examp2} we show that in the case $\beta=\frac nm$, the solution can be unbounded.

\begin{example}\label{examp1}
Let $u_{\alpha}(z)=1-|z|^{\alpha}$, for $0>\alpha>2-\frac {2n}{m}$, be a function defined in the unit ball $B(0,1)\subset \mathbb C^n$, $n\geq 2$. Then $u_{\alpha}$ is $m$-subharmonic, and $\HH(u_{\alpha})=cf_{\alpha}(z)dV_{2n}$. Here, $c$ is a constant, and the density is given by $f_{\alpha}(z)=|z|^{(\alpha-2)m}$. For $\beta\in (1,\frac nm)$, the function $f_{\alpha}$ is in $L^{\beta}(B(0,1))$  if, and only if, $\alpha>2-\frac {2n}{m\beta}$. Finally, $u_{\alpha}\in \mathcal E_{p,m}(B(0,1))$ if, and only if, $p<\frac {-2n+(2-\alpha)m}{\alpha}$. Then, choose $\epsilon>0$ small enough, and let $\alpha(\epsilon)=2-\frac {2n}{m\beta}+\frac{2\epsilon}{m\beta}$. Then $u_{\alpha(\epsilon)}\notin \mathcal E_{p(\epsilon),m}(B(0,1))$ for
\[
p(\epsilon)=\frac {mn(\beta-1)+m\epsilon}{n-m\beta-\epsilon}.
\]
\hfill{$\Box$}
\end{example}

\begin{example}\label{examp2}
Let $0<\alpha<\frac {n-1}{n}$, and define  $u_{\alpha}(z)=(\ln 2)^{\alpha}-\left(-\ln |z|\right)^{\alpha}$. Then $u_{\alpha}$ is a unbounded $m$-subharmonic function defined in the ball $B(0,\frac 12)\subset \mathbb C^n$, $n\geq 2$. We have,
\[
\HH(u_{\alpha})=cf_{\alpha}(z)dV_{2n},
\]
where $c$ is a constant, and the density is given by
\[
f_{\alpha}(z)=|z|^{-2m}\left(-\ln|z|\right)^{\alpha m -m -1}\left((2n-2m)(-\ln |z|)+m(1-\alpha)\right).
\]
Finally, note that $f_{\alpha}\in L^{\frac nm}(B(0,\frac 12))$, since
\[
\int_{B(0,\frac 12)}(f_{\alpha})^{\frac nm}dV_{2n}\leq C\int_0^{\frac 12}\frac {dt}{t(-\ln t)^{n(1-\alpha)}}<\infty.
\]\hfill{$\Box$}
\end{example}

\section{Stability}

 Before proving our stability result, Theorem~\ref{stab}, let us recall some necessary tools and results. For $u,v\in \mathcal E_{p,m}(\Omega)$, $p>0$,  define
\[
\Jp(u,v)=\left(\int_{\Omega}|u-v|^p(\HH(u)+\HH(v))\right)^\frac {1}{p+m},
\]
then it follows from~\cite{ACquasi} that $(\mathcal E_{p,m}(\Omega),\Jp)$ is a complete quasimetric space. We shall as well need the following comparison principle from~\cite{thien2}.

\begin{theorem}\label{cp}
Let $\Omega$ be a bounded $m$-hyperconvex domain, let $u,v\in \mathcal E_{0,m}(\Omega)$($p=0$), $w\in \mathcal E^0_m(\Omega)$, then
\[
\int_{\{u<v\}}(v-u)^m\HH(w)\leq m!\|w\|_{\infty}^{m-1}\int_{\{u<v\}}(-w)(\HH(u)-\HH(v)).
\]
\end{theorem}

\begin{theorem}\label{stab} Let $n\geq 2$, $1\leq m<n$, $1<\beta\leq\frac{n}{m}$, and assume that $\Omega\subset \mathbb C^n$ is a $m$-hyperconvex domain. Furthermore, let that $f_j,f\in L^{\beta}(\Omega)$ be such that $\|f_j-f\|_{\beta}\to 0$, as $j\to \infty$. Then for $p<p_{\infty}$ it holds that $\Jp(\operatorname{U}_{f_j},\UU)\to 0$. Here,
\[
p_{\infty}=\left\{
             \begin{array}{ll}
               \frac{(\beta-1)nm}{n-\beta m}, & \text { if } \ \beta<\frac{n}{m}\\[2mm]
               \infty, & \text { if } \ \beta=\frac{n}{m}.
             \end{array}
           \right.
\]
Moreover, there exist a constant $C$ depending only on $\|f\|_{\beta}$, $\beta$, $m$, $p$, and a constant $\gamma$ depending only on $\beta$, $m$, $p$ such that
\[
\operatorname J_{p}(\operatorname{U}_{f_j},\UU)\leq C\|f_j-f\|_{\beta}^{\gamma}.
\]
The constant $\gamma$ can be taken as
\[
\begin{cases}
\gamma=\frac {p}{m(p+m)} & \text{, when } p\leq\frac {m(\beta-1)}{\beta};\\[2mm]
\gamma\in \left(0,\frac {\beta-1}{(p+m)\beta}\right)=\left(0,\frac {n-m}{n(m+p)}\right) & \text{, when } p>\frac {m(\beta-1)}{\beta}=\frac {m(n-m)}{n} \text{ and }\beta=\frac nm;\\[2mm]
\gamma\in \left(0,\frac {\beta m(n+p)-n(m+p)}{m(p+m)(\beta(n+m)-n)}\right)  & \text{, when } p>\frac {m(\beta-1)}{\beta} \text{ and }\beta<\frac nm.
\end{cases}
\]
\end{theorem}
\begin{proof}
Let $\frac 1{\alpha}+\frac {1}{\beta}=1$, and let $p_0=\frac {m(\beta-1)}{\beta}$. Then, choose $\varphi_0\in \mathcal E_{m}^{0}(\Omega)$ such that $\HH(\varphi_0)=dV_{2n}$. Since $\|f_j-f\|_{\beta}\to 0$, as $j\to \infty$, we can assume
\[
\|f_j-f\|_{\beta}\leq \frac 1{2^j}.
\]
Thanks to Theorem~\ref{cp},
\begin{multline*}
\operatorname J_{p_0}(\operatorname{U}_{f_j},\UU)^{p_0+m}=\int_{\Omega}|\operatorname U_{f_j}-\UU|^{p_0}(\HH(\operatorname U_{f_j})+\HH(\UU))\\
= \int_{\Omega}|\operatorname U_{f_j}-\UU|^{p_0}(f_j+f)dV_{2n}\leq \left(\int_{\Omega}|\operatorname U_{f_j}-\UU|^{p_0\alpha}dV_{2n}\right)^{\frac {1}{\alpha}}\|f_j+f\|_{\beta}\\
\leq\left(m!\|\varphi_0\|^m_{\infty}\int_{\Omega}|f_j-f|dV_{2n}\right)^{\frac {1}{\alpha}}\|f_j+f\|_{\beta}\\
\leq \left(m!\|\varphi_0\|^m_{\infty}\right)^{\frac {\beta-1}{\beta}}(2\|f\|_{\beta}+1)V_{2n}(\Omega)^{\frac {(\beta-1)^2}{\beta^2}}\|f_j-f\|_{\beta}^{\frac {\beta-1}{\beta}}=C\|f_j-f\|_{\beta}^{\frac {p_0}{m}}.
\end{multline*}
If $p<\frac {m(\beta-1)}{\beta}=p_0$, then
\begin{multline*}
\Jp(\operatorname{U}_{f_j},\UU)^{p+m}=\int_{\Omega}|\operatorname U_{f_j}-\UU|^p(\HH(\operatorname U_{f_j})+\HH(\UU)) \\ \leq \operatorname J_{p_0}(\operatorname{U}_{f_j},\UU)^{\frac {p(p_0+m)}{p_0}}\|f_j+f\|_{1}^{\frac {p_0-p}{p_0}},
\end{multline*}
and we can use the estimation above to obtain
\[
\operatorname J_{p}(\operatorname{U}_{f_j},\UU)\leq C\|f_j-f\|_{\beta}^{\frac {p}{m(p+m)}}.
\]

Next, assume that $p_{\infty}>p>p_0$, and choose $p'\in (p,p_{\infty})$. Therefore, from $\operatorname{U}_{f_j},\UU\in \mathcal E_{p',m}(\Omega)$ it follows
\begin{multline}\label{stab1}
\operatorname J_{p}(\operatorname{U}_{f_j},\UU)^{p+m}=\int_{\Omega}|\operatorname U_{f_j}-\UU|^{p}(\HH(\operatorname U_{f_j})+\HH(\UU))\\
\leq\left(\operatorname J_{p_0}(\operatorname{U}_{f_j},\UU)\right)^{\frac {(m+p_0)(p'-p)}{p'-p_0}}\left(\int_{\Omega}|\operatorname U_{f_j}-\UU|^{p'}(\HH(\operatorname U_{f_j})+\HH(\UU))\right)^{\frac {p-p_0}{p'-p_0}}.
\end{multline}
From Theorem~\ref{thm_holder}, and Lemma~\ref{l1}, we get that the second term in~(\ref{stab1}) is bounded by
\begin{multline*}
\int_{\Omega}|\operatorname U_{f_j}-\UU|^{p'}(\HH(\operatorname U_{f_j})+\HH(\UU))\leq e_{p'}(\operatorname U_{f_j}+\operatorname U_{f})\\
\leq D^{\frac {m+p'}{m}}\left(e_{p'}(\operatorname U_{f_j})^{\frac {1}{m+p'}}+e_{p'}(\UU)^{\frac {1}{m+p'}}\right)^{m+p'}
\leq D'\left(\|f_j\|_{\beta}^{\frac {1}{m}}+\|f\|_{\beta}^{\frac {1}{m}}\right)^{m+p'}\\
\leq D'\left((\|f\|_{\beta}+1)^{\frac {1}{m}}+\|f\|_{\beta}^{\frac {1}{m}}\right)^{m+p'}
\end{multline*}
where the constants $D$ and $D'$ do not depend on $\operatorname U_{f_j}$ and $\UU$.
Finally, we obtain
\begin{multline*}
\operatorname J_{p}(\operatorname{U}_{f_j},\UU)\leq C\|f_j-f\|_{\beta}^{\frac {p_0(p'-p)}{m(p+m)(p'-p_0)}}=C\|f_j-f\|_{\beta}^{\frac {(\beta-1)(p'-p)}{(p+m)(\beta p'-m(\beta-1))}}\\ =C\|f_j-f\|_{\beta}^{\gamma}.
\end{multline*}
Passing to the limit with $p'\to p_{\infty}$, we get desired conclusion.
\end{proof}

\section{K\"ahler manifold case}

Let $n\geq 2$, $p\geq 0$, and let $1\leq m< n$. Assume that $(X,\omega)$ is a connected and compact K\"ahler manifold of complex dimension $n$, where $\omega$ is a K\"ahler form on $X$ such that $\int_{X}\omega^n=1$. For any $u\in \mathcal {SH}_m(X,\omega)$, let
\[
\omega_u=dd^cu+\omega.
\]
and the complex Hessian operator is defined by
\[
\operatorname{H}_m(u):=\omega_u^m\wedge \omega^{n-m}.
\]
We define the class of \emph{$(\omega,m)$-subharmonic functions with bounded $(p,m)$-energy} as
\[
\mathcal E_{p,m}(X,\omega):=\left\{u\in \mathcal E_m(X,\omega): u\leq 0, \int_X(-u)^p\operatorname{H}_m(u)<\infty\right\},
\]
where
\[
\mathcal E_m(X,\omega)=\left\{u\in \mathcal {SH}_m(X,\omega): \int_X\operatorname{H}_m(u)=1 \right\}.
\]
Additionally, information on  $(\omega,m)$-subharmonic functions defined on compact K\"ahler manifolds we refer to~\cite{ACchar,DL,GLZ,LN, LN2}.

In this section, we shall arrive in Theorem~\ref{mt_Kahler} to a regularity result on compact K\"ahler manifold corresponding to that of Theorem~\ref{mt}.
With a few minor changes we are able to follow Section~\ref{ms}. The counterpart of Lemma~\ref{l1} is as follows.

\begin{lemma}
If $u\in \mathcal E_{p,m}(X,\omega)\cap \mathcal M_{\beta}$, $p\geq0$, then $u\in \mathcal E_{q,m}(X,\omega)$, for
\[
\begin{cases}
0<q<\frac {(\beta-1)n(1+p)}{\beta(n-m)},  &  \text{ when }\beta> \frac {n(1+p)}{n+mp};\\[2mm]
0<q\leq p, &  \text{ when } \beta\leq \frac {n(1+p)}{n+mp}.
\end{cases}
\]
\end{lemma}

 From~\cite[Theorem 4.4]{ACchar}, it follows that if $u\in \mathcal E_{p,m}(X,\omega)$, then $u\in L^q(X)$, for $q<\frac {(p+1)n}{n-m}$, then again using the reasoning from Section~\ref{ms} we get the counterpart of Lemma~\ref{l2}.

\begin{lemma}
If $u\in \mathcal E_{p,m}(X,\omega)\cap \mathcal M_{\beta}$, then $u\in \mathcal E_{q,m}(X,\omega)$, for
\[
\begin{cases}
0<q<\frac {n(\beta -1)}{n-\beta m}, & \text{ when } \frac {n(1+p)}{n+mp}<\beta<\frac nm; \\
q>0,  & \text{ when } \beta=\frac nm.
\end{cases}
\]
\end{lemma}

Since, if $u\in \mathcal {SH}_m(X,\omega)$, then $u\in L^q(X)$, for $q<\frac n{n-m}$ (\cite[Corollary 6.7]{LN}), we get the analogue of Theorem~\ref{mt}.

\begin{theorem}\label{mt_Kahler}  For $1<\beta\leq\frac nm$, let $u\in \mathcal M_{\beta}\cap \E_m(X,\omega)$. Then the following holds:
\begin{enumerate}\itemsep2mm
\item if $\beta=\frac nm$, then $u\in \mathcal E_{p,m}(X,\omega)$, for all $p>0$. Furthermore, $u\in L^q(X)$, for all $q>0$;

\item if $\beta<\frac nm$, then $u\in \mathcal E_{p,m}(X,\omega)$, for $p<\frac {n(\beta -1)}{n-\beta m}=p_0$.  Furthermore, $u\in L^q(X)$, for all
\[
q<\frac {(p_0+1)n}{n-m}=\frac {\beta n }{n-\beta m}.
\]
\end{enumerate}
\end{theorem}

\end{document}